\documentclass{amsproc}

\usepackage{enumerate}
\usepackage{graphicx}
\usepackage{pstricks}
\usepackage{pst-plot}
\usepackage{pspicture}
\usepackage{mathrsfs}

\theoremstyle{definition}
\newtheorem{theorem}{Theorem}[section]
\newtheorem{lemma}[theorem]{Lemma}

\newtheorem{definition}[theorem]{Definition}

\theoremstyle{remark}

\numberwithin{equation}{section}

\newcommand{\nint}{\int_{-\infty}^{+\infty}}

\newcommand{\var}{\varepsilon}

\newcommand{\m}{$ - $}
\newcommand{\p}{$ + $}
\begin{document}
\title[Uniqueness of standing-waves for pure-power combinations]
{Uniqueness of standing-waves for a
non-linear Schr\"odinger equation with three pure-power
combinations in dimension one}



\author[D.~Garrisi]{Daniele Garrisi}
\address{}
\curraddr{Room 5S167, Building 5\\
Inha University\\
Namgu Inharo 100\\
Incheon 22212 South Korea}
\email{daniele.garrisi@inha.ac.kr}
\thanks{The first author was supported by INHA UNIVERSITY Research
Grant through the project number 51747-01 titled
"Stability in non-linear evolution equations".}

\author[V.~Georgiev]{Vladimir Georgiev}
\address{Vladimir Georgiev\\
Department of Mathematics,
University of Pisa,
Largo Bruno Pontecorvo 5
I - 56127 Pisa,
Italy\\
and \\
Faculty of Science and Engineering, Waseda University,
3-4-1, Okubo, Shinjuku-ku, Tokyo 169-8555 Japan}
\curraddr{}
\email{georgiev@dm.unipi.it}
\thanks{The second author was supported in part by  INDAM, 
GNAMPA - Gruppo Nazionale per l'Analisi Matematica, la 
"Probabilit\`a e le loro Applicazioni" and by Institute of 
Mathematics and Informatics, Bulgarian Academy of Sciences and Top 
Global University Project, Waseda University.}
\date{2017, July 26}
\begin{abstract}
We show that symmetric and positive profiles of ground-state 
standing-wave of the non-linear Schr\"odinger equation are
non-degenerate and unique up to a translation of the argument and 
multiplication by complex numbers in the unit sphere. 
The non-linear term is a combination of two
or three pure-powers. The class of non-linearities satisfying the 
mentioned properties can be extended beyond two or three power 
combinations. Specifically, it is sufficient that an Euler 
differential inequality is satisfied and that a certain auxiliary 
function is such that the first local maximum is also an absolute 
maximum.
\end{abstract}
\maketitle
\section{The role of the uniqueness and 
non-degeneracy in the stability}
A standing-wave is a function defined as 
$ \phi(t,x) := e^{i\omega t} u(x) $,
where $ \omega $ is a real number, 
$ u $ is a complex-valued function in 
$ H^1 (\mathbb{R};\mathbb{C}) $ and $ \phi $ is a solution to
the non-linear Schr\"odinger equation
\begin{equation}
\label{eq.2017-07-27.10:05}
i\partial_t \phi(t,x) + \partial_{xx}^2
\phi(t,x) - F'(\phi(t,x)) = 0,
\end{equation}
The profile of a standing-wave is just $ R(x) := |u(x)| $. The
literature is concerned with the existence and the
stability of standing-waves whose profiles obey prescribed 
variational characterizations. The profiles we are interested
in are minima of the energy functional
\begin{equation*}
E(u) := \frac{1}{2}\nint |u'(x)|^2 dx + \nint F(u(x)) dx
\end{equation*}
on the constrained defined as $ S(\lambda) :=
\{u\in H^1 (\mathbb{R})\mid \|u\|_{L^2}^2 = \lambda\} $ where 
$ \lambda > 0 $. As one can easily check, if $ u $ is a minimum of 
the energy functional, then $ v(x) := zu(x + y) $ belongs to the 
same constraint and has the same energy. Therefore, it is 
a new minimum, for every choice of $ z $ in $ S^1 $ 
(complex numbers in the unit sphere) and $ y $ in $ \mathbb{R} $. 
Then, $ u $ clearly a degenerate critical point of $ E $ on the 
constraint $ S(\lambda) $, as the transformations defined above 
show that $ u $ is the limit of a sequence of critical points. 
Therefore, both uniqueness 
and non-degeneracy need to be defined. We introduce the notation
\begin{equation*}
\mathcal{G}_\lambda := \{u\in S(\lambda)\mid E(u) = 
\inf_{S(\lambda)} E\}.
\end{equation*}
The set we defined is sometimes called \textsl{ground state}, 
as in \cite{BBBM10}, even if the literature occasionally adopts 
this term to address more generally positive solutions to 
semi-linear elliptic equations, \cite{DdPG13}.
We denote by $ H^1 _r (\mathbb{R}) $ the set of real-valued
$ H^1 $ functions
which are radially symmetric with respect to the origin. 
\begin{definition}[Uniqueness and non-degeneracy]
A pair $ (F,\lambda) $ satisfies the uniqueness property if 
given $ u $ and $ v $ in $ \mathcal{G}_\lambda $,
there exists $ (z,y) $ in $ S^1\times\mathbb{R} $ such that
$ u(x) = zv(x + y) $ for every $ x $ in $ \mathbb{R} $. It satisfies
the non-degeneracy property if the function $ E_r $ obtained as a 
restriction of $ E $ on $ S(\lambda)\cap H^1 _r (\mathbb{R}) $ 
has non-degenerate minima.
\end{definition}
Uniqueness and non-degeneracy are not interesting features of the 
energy functional, but also play a role in the orbital stability of
standing-wave solutions to \eqref{eq.2017-07-27.10:05}. We say that 
\eqref{eq.2017-07-27.10:05} is globally well-posed in 
$ H^1 (\mathbb{R};\mathbb{C}) $
if, given $ u_0 $ in $ H^1 (\mathbb{R};\mathbb{C}) $, there exists 
a solution
\[
\phi\colon [0,+\infty)\times \mathbb{R}\to\mathbb{C}
\]
such that $ \phi(0,x) = u_0 (x) $ and the map 
\begin{equation*}
U\colon [0,+\infty)\to H^1 (\mathbb{R};\mathbb{C})\to 
H^1 (\mathbb{R};\mathbb{C}),\quad U_t (u_0) := 
\phi(t,\cdot)
\end{equation*}
is of class
\[
C\sp 1 \big([0,+\infty);H^{-1}(\mathbb{R};\mathbb{C})\big)\cap 
C\big([0,+\infty);H\sp 1 (\mathbb{R};\mathbb{C})\big).
\]
On the set $ H^1 (\mathbb{R};\mathbb{C}) $ we consider the metric induced
by the scalar product
\[
(u,w)_{H^1 (\mathbb{R};\mathbb{C})} :=
\text{Re}\int_{\mathbb{R}} u(x)\overline{w}(x) dx +
\text{Re}\int_{\mathbb{R}} u'(x)\cdot \overline{w'(x)} dx
\]
and denote it by $ d $.
\begin{definition}[Stability]
A subset $ \mathcal{G} $ of $ H\sp 1 (\mathbb{R};\mathbb{C}) $ is 
said stable if for every $ \delta > 0 $ there exists 
$ \varepsilon > 0 $ such 
that $ d(u_0,\mathcal{G}) < \varepsilon\implies 
d(U_t (u_0),\mathcal{G}) < \delta 
$ for every $ t\geq 0 $.
\end{definition}
Given $ u $ in 
$ \mathcal{G}_\lambda $, we define
\begin{equation}
\label{eq.2017-07-27.10:49}
\mathcal{G}_\lambda (u) := 
\{zu(\cdot + y)\mid (z,y)\in S^1\times\mathbb{R}\}.
\end{equation}
In general, if $ u $ is a minimum of $ E $, then 
$ \mathcal{G}_\lambda (u) $ is a subset of the ground state 
$ \mathcal{G}_\lambda $. 
The stability of these two sets is object of interest of the 
literature since the work of T.~Cazenave and P.~L.~Lions,
\cite{CL82}, where pure-powers are considered. Results of 
stability of the ground-state have been extended to
more general non-linearities, as in \cite{BBGM07,Shi14}. We
also mention other references which target the stability 
of the ground-state in other evolutionary equation, as 
multi-constraint non-linear Schr\"odinger systems, 
\cite{Iko14,Bha15,LNW16}, coupled non-linear Schr\"odinger systems
(NLS + NLS), \cite{Oht96,NW11,GJ16}, coupled non-linear 
Schr\"odinger and Kortweg-de Vries equation (NLS + KdV), 
\cite{AA03}, non-linear Klein-Gordon 
equation (NLKG), \cite{BBBM10}, (NLKG + NLKG), \cite{Gar12}.
In most cases, the stability of the ground-state is a consequence 
of the Concentration-Compactness Lemma, \cite{Lio84a,Lio84b}. Coupled 
equations present some additional difficulties (rescalings
do not work) but they can be worked 
around with \textsl{ad hoc} rescalings, as in \cite{AA03} or with
inequalities obtained through symmetric rearrangements for more general 
non-linearities, as \cite[Lemma~3.1]{Gar12} and \cite[Proposition~1.4]{Bye00},
or through the \textsl{coupled rearrangement} defined in \cite[\S2.2]{Shi13}. 
\vskip .5em

The stability of 
$ \mathcal{G}_\lambda (u) $ is more challenging than the stability 
of $ \mathcal{G}_\lambda $: there might be
solutions to \eqref{eq.2017-07-27.10:05} with initial values 
close to $ \mathcal{G}_\lambda (u) $, but intermediate values far
from it. Another application of the Concentration-Compactness Lemma
and the stability of the ground state implies that these
intermediate values are close to another set 
$ \mathcal{G}_\lambda (v) $ (as shown in 
\textsc{Figure}~\ref{fig.2017-07-27.15:17}). 
A simple way to rule out the existence of these trajectories
is to prove that there is only
one $ \mathcal{G}_\lambda (u) $, as $ u $ varies in 
$ \mathcal{G}_\lambda $. This is the approach followed in \cite{CL82} 
with the help of a uniqueness result, \cite{MS87},
which specifically applies to pure-powers.
Therefore, $ \mathcal{G}_\lambda = \mathcal{G}_\lambda (u) $, and the 
second set is stable because the first one is stable. 
Another way is to show
that there are only finitely many of these sets 
$ \mathcal{G}_\lambda (u) $.
In this case (see \textsc{Figure}~\ref{fig.2017-07-27.15:17}), 
trajectories bridging two different sets need to achieve a minimum 
amount of energy, which is too
high if the initial value is too close to 
$ \mathcal{G}_\lambda (u) $, as it follows from \cite[\S4]{GG17}. 
Now, from the work of L.~Jeanjean 
and J.~Byeon, \cite{BJM09}, in every set 
$ \mathcal{G}_\lambda (u) $
there exists a unique positive $ R $ in $ H^1 _r $. Therefore,
the problem of the stability of the set \eqref{eq.2017-07-27.10:49}
reduces to showing that $ \mathcal{G}_{\lambda,r} := \mathcal{G}_\lambda\cap H^1 _r $ is finite, \cite[Proposition~5]{GG17}. And this 
follows straightforwardly from the non-degeneracy of minima 
of $ E_r $. From \cite[Corollary~2]{GG17}, the uniqueness holds if 
$ \mathcal{G}_{\lambda,r} ^+ := 
 \mathcal{G}_{\lambda}\cap H^1 _{r,+} $ is a singleton.
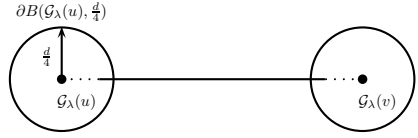
\begin{figure}
\caption{Trajectories bridging two different sets}
\label{fig.2017-07-27.15:17}
\begin{pspicture}[showgrid=false](0,-1)(4,1.5)
\psunit=10pt
\pscircle(0,0){2}
\pscircle(4,0){2}
\psdots(0,0)(4,0)
\psline{->}(0,0)(0,0.7)
\rput(-0.2,0.3){\scalebox{0.6}{$ \frac{d}{4} $}}
\rput(0.2,-0.3){\scalebox{0.6}{$ \mathcal{G}_{\lambda} (u) $}}
\rput(4.2,-0.3){\scalebox{0.6}{$ \mathcal{G}_{\lambda} (v) $}}
\rput(0,0.9){\scalebox{0.6}{$ \partial B(\mathcal{G}_\lambda (u),
\frac{d}{4}) $}}
\psline[linestyle=dotted](0,0)(0.5,0)
\psline(0.5,0)(3.5,0)
\psline[linestyle=dotted](3.5,0)(4,0)
\end{pspicture}
\end{figure}
\section{Assumptions on $ F $ and 
non-degeneracy}
The non-linearity $ F $ is a $ C^2 $ real valued function defined 
on $ \mathbb{C} $; $ F(s) = G(|s|) $ for every $ s $ in 
$ \mathbb{C} $. We list our assumptions trying to keep the notation
consistent with \cite{GG17}:
\begin{gather}
\label{G1}
\tag{G1}
\exists s_0 > 0\text{ such that } G(s_0) < 0\\
\tag{G2b}
\label{G2b}
-C|s|^{p^*}\leq G(s),\quad s\geq s_*,\quad 2 < p_* < 6\\
\tag{G4'}
\label{G4}
G(0) = G'(0),\quad
|G''(s)|\leq C(|s|^{p - 2} + |s|^{q - 2}),\quad 2 < p < q
\end{gather}
are satisfied. 
\eqref{G4} makes 
sure that the energy functional $ E $ is 
$ C^2(H^1(\mathbb{R};\mathbb{C}),\mathbb{R}) $, which is 
a consequence of regularity theorems on Nemytski operators
proved in \cite{AP93}; \eqref{G2b} makes $ E $ a 
coercive functional on $ S(\lambda) $ and provides a-priori 
estimates, and thus the global well-posedness of 
\eqref{eq.2017-07-27.10:05}. For \eqref{G1} we refer to 
\cite{BBGM07}; together with \eqref{G2b}, it ensures that a 
minimum of $ E $ on $ S(\lambda) $ exists, if $ \lambda $ is large 
enough, see \cite[Theorem~2]{BBGM07}. In the quoted reference,
a stronger condition than \eqref{G2b} is actually used,
by setting $ s_* $ is to zero. However, their proof applies
under the weaked assumption as well. 
We illustrate how the non-degeneracy of 
minima of $ E $ on $ S(\lambda) $ are obtained in \cite{GG17}. 
We set 
$ S_r (\lambda) := S(\lambda)\cap H^1 _r (\mathbb{R};\mathbb{R}) $.
Suppose that $ R_0 $ is a minimum of $ E_r $ on $ S_r (\lambda) $. 
Then, there exists $ \omega_0 $ such that
\begin{equation}
\label{eq.5}
R_0 '' - G'(R_0) - \omega_0 R_0 = 0.
\end{equation}
By looking at conserved quantities, as in 
\cite[Proposition~3]{GG17}, there holds $ \omega_0 > 0 $.
In order to prove that $ R_0 $ is a non-degenerate minimum
of $ E_r $, it is enough to consider a function $ v $ in $ H^1 _r $
such that $ (v,R_0)_2 = 0 $ and show that
there exists $ C $ not depending on $ v $ such that
\begin{equation}
\label{eq.2017-07-28.13:58}
D^2 E(R_0)[v,v]\geq C\|v\|^2 _{H^1}.
\end{equation}
In order to evaluate the Hessian, it only takes to define 
a smooth curve $ \alpha $ from $ (-\varepsilon,\varepsilon) $
to $ S_r (\lambda) $ such that $ \alpha(0) = R_0 $, 
$ \alpha'(0) = v $. Then
\begin{equation*}
(E\circ\alpha)''(0) = D^2 E(R_0)[v,v] = \nint \Big(|v'(x)|^2  +
(G''(R_0(x)) + \omega_0) v(x)^2\Big) dx =: \xi(v).
\end{equation*}
Since the functional above is homogeneous, we only need to
to show that the infimum of $ \xi $ is positive when restricted
to the unit sphere $ S_r (1) $ in $ L^2 $. Then 
\eqref{eq.2017-07-28.13:58} will follow from the Banach-Steinhaus
theorem. The functional $ \xi $ is certainly non-negative, 
because the fact that $ R_0 $ is a minimum is part of our 
assumptions. 
It is convenient to prove that $ \xi $ does actually achieve 
its infimum.
We use as a reference the proof of \cite[Proposition~2.9]{Wei85}.
In fact, although that deals specifically with pure-powers,
it can be applied to more general non-linearities, provided
$ G''(0) = 0 $. Let $ v_0 $ be a minimum of $ \xi $. If 
$ \xi(v) = 0 $, then there exists $ \beta $ in $ \mathbb{R} $
such that
\begin{equation*}
L_+ (v) := -v'' + G''(R_0)v + \omega_0 v = \beta R_0.
\end{equation*}
From $ R_0 $ we construct a one-parameter family of solutions
starting from $ R_0 $. We premise a few remarks. Firstly,
$ R_0(0) $ is a solution to the equation
\begin{equation}
\label{eq.2017-07-28.12:07}
V(R_0 (0)) = \omega_0,\quad V(s) := -\frac{2G(s)}{s^2}
\end{equation}
which is the auxiliary function mentioned in the abstract of
this paper. Secondly, from \cite[Theorem~5]{BL83a}, $ R_0 (0) $ is 
the least positive solution to \eqref{eq.2017-07-28.12:07}.
Moreover, $ R_0 $ is also an even function decreasing
on $ [0,+\infty) $. Therefore, $ R_0 '' (0) < 0 $. Then 
\begin{equation*}
R_0 (0) = \inf\{s > 0\mid V(s) = \omega_0\},\quad V'(R_0 (0)) > 0.
\end{equation*}
The second inequality is obtained by combining two equalities
which in turn can be obtained by multiplying \eqref{eq.5} by $ R_0 $
and $ R_0 ' $, as in \cite[Proposition~4]{GG17}.
The construction of the one-parameter family is made as follows:
the function 
\begin{equation*}
R_* (\omega) := \inf\{s > 0\mid V(s) = \omega\}
\end{equation*}
is smooth in a neighborhood of $ \omega_0 $, because 
$ V'(R_0 (0)) > 0 $. We define the function $ R_\omega $
as solution to the initial value problem
\begin{equation}
\label{eq.2017-07-27.13:21}
R_\omega '' (x) - G'(R_\omega (x)) - \omega R_\omega(x) = 0
\quad R_\omega '(0) = 0,\quad R_\omega(0) = R_* (\omega).
\end{equation}
We set $ S(\omega,x) = \frac{\partial R}{\partial \omega} 
(\omega,x) $
and define $ S(\omega_0,x) := S_0 (x) $.
Therefore, taking the derivative with respect to $ \omega $
in \eqref{eq.2017-07-27.13:21}, and evaluating at 
$ \omega = \omega_0 $, we obtain
\begin{equation*}
L_+ (S_0) = R_0.
\end{equation*}
Taking the $ L^2 $ scalar product with $ R_0 $, we obtain
\[
\frac{1}{2}\frac{d}{d\omega}\|R(\omega_0,\cdot)\|_2 ^2 = 
(L^+ (S_0),R_0)_2.
\]
Up to a sign-change, the quantity appearing in the left term is 
the one denoted by $ Q(\omega) $ in the paper of 
N.~G.~Vakhitov and A.~A.~Kolokolov \cite{VK73}.
Therefore, it is worth to investigate the behavior of
the derivative of the function 
$ \lambda(\omega) := \|R(\omega,\cdot)\|_2 ^2 $ at the point
$ \omega = \omega_0 $. The calculations made in \cite[\S4]{GG17}
can be summarized as follows: there exists a positive function 
$ \Psi $ such that
\begin{equation}
\label{eq.2017-07-31.14:47}
\frac{d\lambda}{d\omega}(\omega_0) = -\frac{2R_* '(\omega_0)}%
{R_*(\omega_0)^5}\int_0\sp 1
\frac{\theta^2 (K(R_* (\omega_0)) - K(\theta R_* (\omega_0)))}
{(\Psi(\theta,R_*(\omega_0),\omega_0))^{3/2}}d\theta\geq 0,
\end{equation}
where $ K(s) = \frac{1}{s^2}(-6G(s) + sG'(s)) $. At this point,
provided $ K $ is a strictly non-decreasing function,
we have $ \lambda'(\omega_0) > 0 $. Since
\begin{equation*}
K'(s) = \frac{12sG(s) - 7s^2G'(s) + s^3 G''(s)}{s^4}
\end{equation*}
this computation suggests to require that
$ 12G(s) - 7sG'(s) + s^2 G''(s) > 0 $ for every $ s $ in the
interval $ (0,R_* (\omega_0)) $. In fact, there is no need
to have a strict inequality here: since the integrand in
\eqref{eq.2017-07-31.14:47} is non-negative, 
if $ \lambda'(\omega_0) $
vanishes, then $ 12G(s) - 7sG'(s) + s^2 G''(s) = 0 $ on
$ (0,R_* (\omega_0)) $, which means that on this interval $ G $
is a linear combination of $ s^2 $ and $ s^6 $. However, the
coefficient of $ s^2 $ is zero, by \eqref{G4},
while the coefficient of $ s^6 $ is equal to zero because it is the 
pure-power critical case where minima of $ E $ over $ S(\lambda) $ 
do not
exist, see \cite[Proof~of~Lemma~3.1]{GG17}. 
Finally, since we wish
to address all the minima, regardless of the constraint, the 
set where the requirement holds should apply to the images
of all the minima. We define
\begin{equation*}
\Omega := \bigcup_{\lambda > 0}\bigcup_{R\in\mathcal{G}_\lambda} 
\text{Img}(R).
\end{equation*}
By \cite[Proposition~4]{GG17}, 
$ \Omega = (0,+\infty) $ if $ V $ is not bounded or $ V $ is 
bounded but $ \sup(V) $ is not achieved. Otherwise,
$ \Omega = (0,R_* (\max(V)) $. Therefore, in \cite{GG17}
we required
\begin{equation}
\label{G3}
\tag{G3}
L(s) := 12G(s) - 7sG'(s) + s^2 G''(s)\geq 0 \text{ on } \Omega.
\end{equation}
There are several non-linearities satisfying the condition
above, starting from pure-powers $ G(s) = -as^p $ with $ a > 0 $ and 
$ 2 < p\leq 6 $. Another example is the combined pure-power
$ G(s) = -as^p + bs^q $ with $ a,b > 0 $ and $ p < q $; 
clearly, in the latter case, 
\begin{equation*}
L(s) = a(p - 2)(6 - p)s^p - b(q - 2)(6 - q) s^q
\end{equation*}
might changes sign. However, the function is non-negative on 
$ \Omega $ which is a bounded interval for this choice of $ G $.
In fact, \eqref{G3} is satisfied, \cite[Corollary~2]{GG17}.
\section{Uniqueness of standing-waves}
\label{s.uniqueness}
The idea of how we obtain the uniqueness of standing-waves is the 
following: if there are two minima $ R_0 $ 
and $ R_1 $ belonging to the same constraint $ S(\lambda) $, 
we consider 
the corresponding Lagrange multipliers $ -\omega_0 $ and
$ -\omega_1 $. From \eqref{eq.2017-07-31.14:47}, 
the function $ \lambda $ is injective
on $ [\omega_0,\omega_1] $, which implies that $ \lambda $ 
is constant, because achieves the same values at the endpoints.
Then $ L \equiv 0 $ on $ (0,R_*(\omega_1)) $ which
implies that $ G $ is a linear combination of $ s^2 $ and $ s^6 $
and gives a contradiction with the sub-critical assumptions. The
only thing we need to take care of is the definition
of $ \lambda $, which is smooth as long as
$ R_* $ is smooth. In turn $ R_* $ is smooth
on $ \omega $ if $ V'(R_* (\omega))\neq 0 $.
Therefore, critical points of $ V $
represent potential discontinuities of the function $ R_* $.
However, $ R_* $ is continuous everywhere if, for instance,
$ V $ does not have local maxima or the first local maximum
is an absolute maximum. Therefore, we set
\begin{equation*}
A := \{s > 0\mid s\text{ is a local maximum of } V\}.
\end{equation*}
The assumption introduced in \cite{GG17} reads
\begin{equation}
\label{G5}
\tag{G5}
A = \emptyset\text{ or } (
A\neq\emptyset,\text{ } V\text{ is bounded and }
V(\inf(A)) = \sup(V) < +\infty).
\end{equation}
To summarize, condition \eqref{G3} allows to state that
the set $ \mathcal{G}_{\lambda,r}^+ $ is finite. If 
\eqref{G5} holds as well, 
then $ \mathcal{G}_{\lambda,r}^+ $ is a singleton,
\cite[Theorem~1.4]{GG17}.
\begin{theorem}[{\cite[Theorem~1.4]{GG17}}]
If the conditions \eqref{G1}, \eqref{G2b}, \eqref{G3},
\eqref{G4} and \eqref{G5} hold, then 
$ \mathcal{G}_\lambda\cap H^1 _r $ consists of exactly two functions, 
$ R_+ $ and $ R_- $. The first is positive while $ R_- = -R_+ $.
\end{theorem}
The assumption (G2a) in \cite{GG17} has been omitted here, as
it can be replaced by \eqref{G4}. This explains the slight
difference with the referenced theorem.
We consider
\begin{equation*}
G(s) = \var_a a s^p + \var_b b s^q + \var_c c s^r,\ 
\{\var_a,\var_b,\var_c\}\subseteq\{-1,0,1\},\ a,b,c > 0,\
2 < p < q < r
\end{equation*}
and discuss the assumptions mentioned above. 
In the remainder of the paper we will describe
the behavior of the two properties for pure-powers, combined 
pure-power, and three pure-power combinations. Some cases 
have already been illustrated in \cite[\S5]{GG17}, but 
we included them for the sake of completeness.
We will leave out the cases 
$ \{\var_a,\var_b,\var_c\}\subseteq\{0,1\} $
as \eqref{G1} is not fulfilled. \eqref{G4} follows from the
fact that all the exponents are bigger than two. 
When the coefficient of highest order term at infinity is positive
\eqref{G2b} is satisfied.
\subsection{Pure-powers}
\label{ss.pp}
If $ G(s) = -as^p $, then \eqref{G1} is satisfied because $ G < 0 $
and \eqref{G2b} holds if $ p < 6 $. Then the function
$ L(s) = a(p - 2)(6 - p)s^p $ is non-negative, while 
$ V = 2as^{p - 2} $ does not have 
local maxima, implying that \eqref{G5} is satisfied. 
\subsection{Combined pure-powers}
\label{ss.cp}
Firstly, we consider the case $ G(s) = -as^p + bs^q $ which
clearly achieves negative values. 
The function $ V $ is bounded and has a single local 
maximum. Therefore, \eqref{G5} is satisfied and \eqref{G3} is 
satisfied
if the (unique) zero of $ L $ occurs \textsl{before} the local
maximum of $ V $, which is the unique zero of $ V' $. We will show that
$ V'(s_0) = 0 $ implies $ L(s_0) > 0 $. In fact, $ V'(s_0) = 0 $
gives
\begin{equation*}
2a(p - 2)s_0 ^{p - 3} - 2b(q - 2)s_0 ^{q - 3} = 0.
\end{equation*}
If we multiply it by $ s_0 ^3 $, and substitute 
$ 2b(q - 2) s_0 ^q $ with $ 2a(p - 2)s_0 ^p $ in $ L $,
we obtain $ L(s_0) = a(p - 2)(q - p)s_0^p > 0 $ which implies
\eqref{G3}. In fact, no sub-critical assumption (which was
required in \cite[\S5]{GG17}) is needed.
If $ G = -as^p - bs^q $ then $ G < 0 $ which implies \eqref{G1}. 
For \eqref{G2b} to hold, we need $ q < 6 $. Then
$ L > 0 $ on $ (0,+\infty) $. \eqref{G5} is satisfied because
$ V' > 0 $ on $ (0,+\infty) $, so $ A = \emptyset $. 
Finally, if $ G(s) = as^p - bs^q $, $ q < 6 $. Since $ V' $
goes to $ +\infty $, $ \Omega = (0,+\infty) $. However,
$ L $ clearly changes sign in a neighborhood of the origin.
Therefore \eqref{G3} is not satisfied.
\subsection{Three pure-power combinations}
The cases with three negative coefficients are ruled out as 
in \ref{ss.cp}. Then all the assumptions are fulfilled.
\eqref{G1} can be easily checked except for the
case $ (\var_a,\var_b,\var_c) = (1,-1,1) $ in \S\ref{sss.pmp}.
The following remark will be useful 
in \S\ref{sss.pmp} and \S\ref{sss.mpm}: given the function
\begin{equation*}
k(s) := A - Bs^{q - p} + Cs^{r - q}
\end{equation*}
with $ p < q < r $ and $ A,B,C $ positive real numbers there holds
\begin{equation}
\label{eq.2}
\inf(k)\geq 0\iff 
A\geq B^{\frac{r - p}{r - q}} C^{\frac{p - q}{r - q}} d_*
\end{equation}
where 
\begin{equation}
\label{eq.6}
d_* := 
\left[\left(\frac{q - p}{r - p}\right)^{\frac{q - p}{r - p}} - 
\left(\frac{q - p}{r - p}\right)^{\frac{r - p}{r - q}}\right] > 0.
\end{equation}
It is obtained by evaluating $ k $ on its unique minimum, obtained
by solving explicitly $ k' = 0 $.
\subsubsection{$ G(s) = -as^p - bs^q + cs^r $}
\label{sss.mmp}
Clearly \eqref{G5} holds, because the set $ A $ is a 
singleton.
If $ p\geq 6 $, then $ L < 0 $ in a neighborhood of the origin.
Therefore \eqref{G3} does not hold because $ \Omega $ contains
small neighborhoods of the origin, as shown in Figure~\ref{fig.mmp}. 
For the case $ p < 6 $
it is convenient to divide $ V' $ and
$ L $ by the leading coefficient, $ s^{q - p} $, 
and use the substitution $ t = s^{q - p} $. As in 
\S\ref{ss.cp}, we need to know the behavior of $ L $ at the
unique zero of $ V' $. We set
\begin{align}
g(t) &:= 1 + \frac{b(q - 2)(6 - q)}{a(p - 2)(6 - p)} t - 
\frac{c(r - 2)(6 - r)}%
{a(p - 2)(6 - p)}t^{\frac{r - p}{q - p}}\\
h(t) &:= 1 + \frac{b(q - 2)}{a(p - 2)} t - \frac{c(r - 2)}{a(p - 2)}
t^{\frac{r - p}{q - p}}.
\end{align}
Let $ t_0 $ be the unique zero of $ h $. 
From $ h(t_0) = 0 $ we obtain
\begin{equation*}
\frac{c(r - 2)}{a(p - 2)} t_0 ^{\frac{r - p}{q - p}} = 
\left(\frac{b(q - 2)}{a(p - 2)} t_0 + 1\right).
\end{equation*}
Then
\begin{equation*}
\begin{split}
g(t_0) &= 
1 + \frac{b(q - 2)(6 - q)}{a(p - 2)(6 - p)} t_0 - 
\frac{6 - r}{6 - p}\left(
\frac{b(q - 2)}{a(p - 2)} t_0 + 1\right) \\
&= \frac{b(q - 2)(r - q)}{a(p - 2)(6 - p)} t_0 + 1 -
\frac{6 - r}{6 - p} > 0.
\end{split}
\end{equation*}
Then \eqref{G3} holds. The behavior of $ g $ and
$ h $ is represented in Figure~\ref{fig.mmp.2}.
\begin{center}
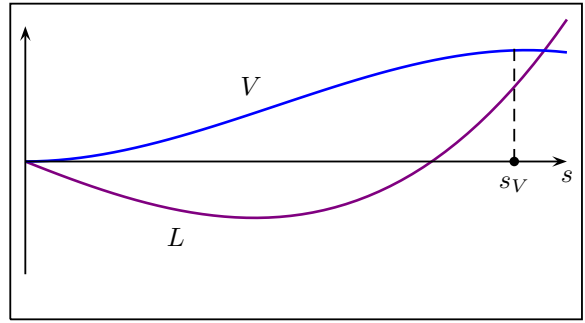
\begin{figure}
\caption{$ L < 0 $ on a subset of $ \Omega = (0,s_V) $}
\label{fig.mmp}
\begin{pspicture}[showgrid=false](0,-2)(7.2,2.5)
\psline[linewidth=0.7pt](-0.2,-2.1)(7.4,-2.1)(7.4,2.1)(-0.2,2.1)(-0.2,-2.1)
\rput(6.5,-0.3){$ s_V $}
\psdot(6.5,0)
\psplot[algebraic,linewidth=1pt,linecolor=violet]{0}{7.2}%
{0.01*(-40*x + 2*x^2 + x^3)}
\psplot[algebraic,linewidth=1pt,linecolor=blue]{0}{7.2}%
{0.1*(x^2 - 0.1*x^3)}
\psline[linewidth=0.7pt,arrowsize=4pt]{->}(0,0)(7.2,0)
\psline[linewidth=0.7pt,arrowsize=4pt]{->}(0,-1.5)(0,1.8)
\psline[linewidth=0.7pt,linestyle=dashed](6.5,0)(6.5,1.5)
\rput(7.2,-0.2){$ s $}
\rput(3,1){$ V $}
\rput(2,-1){$ L $}
\end{pspicture}
\end{figure}
\end{center}
\subsubsection{$ G(s) = -as^p + bs^q + cs^r $}
\label{sss.mpp}
\eqref{G5} always holds as $ V $ has a single local maximum. 
If $ p = 6 $, then $ L > 0 $ everywhere. For $ p\neq 6 $
we can define the functions $ g $ and $ h $
in the same fashion as in \S\ref{sss.mmp}
\begin{align}
g(t) &:= 1 - \frac{b(q - 2)(6 - q)}{a(p - 2)(6 - p)} t - 
\frac{c(r - 2)(6 - r)}%
{a(p - 2)(6 - p)}t^{\frac{r - p}{q - p}}\\
h(t) &:= 1 - \frac{b(q - 2)}{a(p - 2)} t - \frac{c(r - 2)}%
{a(p - 2)}
t^{\frac{r - p}{q - p}}.
\end{align}
If $ p < 6 $ then $ g\geq h $ because each coefficient
of $ g $ is larger than the corresponding coefficient of 
$ h $. Therefore, the first zero of $ L $ occurs after
the first zero of $ V' $, and \eqref{G3}
holds, Figure~\ref{fig.mmp}.
If $ p > 6 $, then $ L $ is negative in a 
neighborhood of
the origin, therefore \eqref{G3} does not hold.
\begin{center}
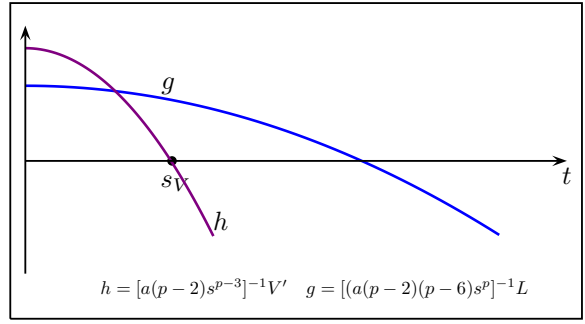
\begin{figure}
\caption{The zero of $ L $ occurs after the first zero of
$ V' $, from \S\ref{sss.mmp}.}
\label{fig.mmp.2}
\begin{pspicture}[showgrid=false](0,-2)(7.2,2.5)
\psline[linewidth=0.7pt](-0.2,-2.1)(7.4,-2.1)(7.4,2.1)%
(-0.2,2.1)(-0.2,-2.1)
\rput(2,-0.3){$ s_V $}
\psdot(1.95,0)
\psplot[algebraic,linewidth=1pt,linecolor=blue]{0}{6.3}%
{0.1*(10 - 0.5*x^2)}
\psplot[algebraic,linewidth=1pt,linecolor=violet]{0}{2.5}%
{0.05*(30 - 8*x^2)}
\psline[linewidth=0.7pt,arrowsize=4pt]{->}(0,0)(7.2,0)
\psline[linewidth=0.7pt,arrowsize=4pt]{->}(0,-1.5)(0,1.8)
\rput(7.2,-0.2){$ t $}
\rput(1.9,1){$ g $}
\rput(2.6,-0.8){$ h $}
\rput[l](1,-1.7){\scalebox{0.7}{%
$ h = [a(p - 2)s^{p - 3}]^{-1} V' $\hskip 1em
$ g = [(a(p - 2)(p - 6)s^p]^{-1} L $}}
\end{pspicture}
\end{figure}
\end{center}
\subsubsection{$ G(s) = as^p - bs^q - cs^r $}
\label{sss.pmm}
For \eqref{G2b} to hold, $ r < 6 $ must be satisfied. 
$ \Omega = (0,+\infty) $, while $ \inf(L) < 0 $. Then
\eqref{G5} holds, but \eqref{G3} does not.
\subsubsection{$ G(s) = as^p - bs^q + cs^r $}
\label{sss.pmp}
For \eqref{G1} to hold we need $ \inf(G) < 0 $. If we set
$ k := [as^p]^{-1} G $, the equivalence \eqref{eq.2} gives
\begin{equation}
\label{eq.2017-08-26.10:59}
a < b^{\frac{r - p}{r - q}} c^{\frac{p - q}{r - q}} d_*.
\end{equation}
If $ p \leq 6 $, then \eqref{G3} does not hold, because
$ L $ is negative in a neighborhood of the origin, as in
Figure~\ref{fig.mmp}. Before looking at the case $ p > 6 $
it is useful to observe that from \eqref{G1}
we have $ \sup(V') > 0 $. On the contrary, $ \inf(-V')\geq 0 $.
We apply \eqref{eq.2} to $ k := [-a(p - 2)s^{p - 3}]^{-1} V' $ 
and obtain
\begin{equation}
\label{eq.2017-08-26.10:58}
a(p - 2)\geq [b(q - 2)]^{\frac{r - p}{r - q}} 
[c(r - 2)]^{\frac{p - q}{r - q}} d_*.
\end{equation}
Dividing term-wise \eqref{eq.2017-08-26.10:58} by 
\eqref{eq.2017-08-26.10:59}, we obtain
\begin{equation}
\label{eq.2017-08-26.11:04}
p - 2 > (q - 2)^{\frac{r - p}{r - q}} 
(r - 2)^{\frac{p - q}{r - q}}.
\end{equation}
By exponentiating both terms to $ r - q $, dividing
by $ (p - 2)^{r - q} $ and applying the variable changes
$ x = p - 2 $, $ y = q - 2 $ and $ z = r - 2 $, 
\eqref{eq.2017-08-26.11:04} reads $ M(x,y,z) > 1 $ which
contradicts Lemma~\ref{lem.1}. Then $ \sup(V') > 0 $ and
\eqref{G5} holds too.

When $ p > 6 $ we need to compare $ L $ and $ V' $. 
Since $ \sup(V') > 0 $, it has two distinct zeroes. 
We will show that $ L $ is negative in the first zero of 
$ V' $, as in Figure~\ref{fig.pmp}. We set
\begin{align}
g(t) &:= 1 - \frac{b(q - 2)(6 - q)}{a(p - 2)(6 - p)} t +
\frac{c(r - 2)(6 - r)}%
{a(p - 2)(6 - p)}t^{\frac{r - p}{q - p}}\\
h(t) &:= 1 - \frac{b(q - 2)}{a(p - 2)} t + \frac{c(r - 2)}{a(p - 2)}
t^{\frac{r - p}{q - p}}.
\end{align}
We call $ t_1 $ the first zero of $ h $. Since $ hV' < 0 $
on $ (0,+\infty) $, $ h(t_1) = 0 $ and $ h'(t_1) < 0 $. 
From $ h(t_1) = 0 $, we obtain
\begin{equation}
\label{eq.2017-08-26.09:48}
\frac{c(r - 2)}{a(p - 2)} t_1 ^{\frac{r - p}{q - p}} = 
\frac{b(q - 2)}{a(p - 2)} t_1 - 1
\end{equation}
which we can substitute into the inequality 
$ t_1 h'(t_1) < 0 $ and obtain
\begin{equation}
- \frac{b(q - 2)}{a(p - 2)} t_1 + 
\frac{r - p}{q - p} \left(\frac{b(q - 2)}{a(p - 2)} t_1 - 1\right) < 0
\end{equation}
which gives
\begin{equation}
\label{eq.2017-08-25.17:49}
t_1 < \frac{a(p - 2)(r - p)}{b(q - 2)(r - q)}.
\end{equation}
Therefore, from \eqref{eq.2017-08-26.09:48} and 
\eqref{eq.2017-08-25.17:49}
\begin{equation*}
\begin{split}
g(t_1)&= 1 - \frac{b(q - 2)(6 - q)}{a(p - 2)(6 - p)} t_1 
+ \frac{6 - r}{6 - p}
\left(\frac{b(q - 2)}{a(p - 2)} t_1 - 1\right)\\
&= \frac{b(q - 2)(r - q)}{a(p - 2)(p - 6)} t_1 - \frac{r - p}{p - 6}\\
&< \frac{b(q - 2)(r - q)}{a(p - 2)(p - 6)}\cdot
\frac{a(p - 2)(r - p)}{b(q - 2)(r - q)} - \frac{r - p}{p - 6} = 0.
\end{split}
\end{equation*}
Then, regardless of the values of the exponents, \eqref{G3} is
never met, while \eqref{G5} holds.
\begin{center}
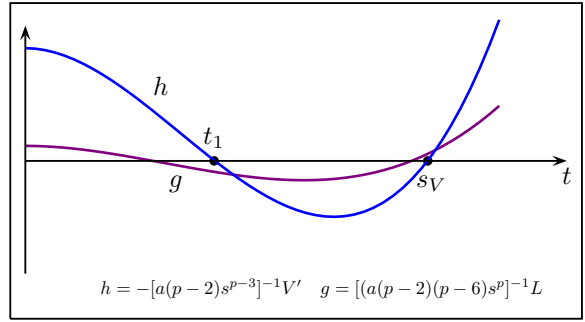
\begin{figure}
\caption{The zero of $ L $ occurs before the first zero of
$ V' $, from \S\ref{sss.pmp}.}
\label{fig.pmp}
\begin{pspicture}[showgrid=false](0,-2)(7.2,2.5)
\psline[linewidth=0.7pt](-0.2,-2.1)(7.4,-2.1)(7.4,2.1)(-0.2,2.1)(-0.2,-2.1)
\rput(2.51,0.3){$ t_1 $}
\psdot(2.51,0)
\rput(5.4,-0.3){$ s_V $}
\psdot(5.35,0)
\psplot[algebraic,linewidth=1pt,linecolor=violet]{0}{6.3}%
{0.01*(20 - 10*x^2 + 1.8*x^3)}
\psplot[algebraic,linewidth=1pt,linecolor=blue]{0}{6.3}%
{0.05*(30 - 8*x^2 + 1.3*x^3)}
\psline[linewidth=0.7pt,arrowsize=4pt]{->}(0,0)(7.2,0)
\psline[linewidth=0.7pt,arrowsize=4pt]{->}(0,-1.5)(0,1.8)
\rput(7.2,-0.2){$ t $}
\rput(1.8,1){$ h $}
\rput(2,-0.3){$ g $}
\rput[l](1,-1.7){\scalebox{0.7}{%
$ h = -[a(p - 2)s^{p - 3}]^{-1} V' $\hskip 1em
$ g = [(a(p - 2)(p - 6)s^p]^{-1} L $}}
\end{pspicture}
\end{figure}
\end{center}
\subsubsection{$ G(s) = as^p + bs^q - cs^r $}
\label{sss.ppm}
\eqref{G2b} implies $ r < 6 $. Therefore, $ \inf(L) < 0 $
on $ (0,\var) $ for $ \var $ suitably small and \eqref{G3}
does not hold, as in Figure~\ref{fig.mmp}.
\subsubsection{$ G(s) = -as^p + bs^q - cs^r $}
\label{sss.mpm}
The conclusions we reached so far depend only on the exponents
and not on the coefficients, as long as their signs are 
prescribed. This case is an exception. 
Firstly, \eqref{G2b} forces $ r < 6 $. 
Since $ V $ is not bounded, $ \Omega = (0,+\infty) $.
If we apply \eqref{eq.2}
with $ k := [a(p - 2)(6 - p)s^p]^{-1} L $, we need
\begin{equation}
\label{eq.1}
a(p - 2)(6 - p)\geq b^{\frac{r - p}{r - q}} 
(q - 2)^{\frac{r - p}{r - q}} 
(6 - q)^{\frac{r - p}{r - q}}
\cdot c^{\frac{p - q}{r - q}}
(r - 2)^{\frac{p - q}{r - q}} 
(6 - r)^{\frac{p - q}{r - q}}
d_*
\end{equation}
where $ d_* $ has been defined in \eqref{eq.6}. \eqref{eq.1}
can be true or false depending on whether
$ a $ is large or small, respectively. 
Therefore, in this section we just show that
whenever \eqref{G3} is satisfied, \eqref{G5} is satisfied too.
In fact, if \eqref{G5} does not hold, then $ \inf(V') < 0 $.
By applying \eqref{eq.2} with $ k := [a(p - 2)s^{p - 3}]V' $,
we obtain
\begin{equation}
\label{eq.3}
a(p - 2) < b^{\frac{r - p}{r - q}} 
(q - 2)^{\frac{r - p}{r - q}} 
c^{\frac{p - q}{r - q}}
(r - 2)^{\frac{p - q}{r - q}} 
d_*.
\end{equation}
We divide term-wise \eqref{eq.3} by
\eqref{eq.1} and obtain
\begin{equation}
\label{eq.4}
(6 - p) > (6 - q)^{\frac{r - p}{r - q}}
(6 - r)^{\frac{p - q}{r - q}}.
\end{equation}
By exponentiating both terms to $ r - q $, dividing
by $ (6 - p)^{r - q} $ and applying the variable changes
$ x = 6 - p $, $ y = 6 - q $ and $ z = 6 - r $, 
\eqref{eq.4} reads $ M(x,y,z) > 1 $ which
contradicts Lemma~\ref{lem.1}.
\vskip .5em

We give a proof of the lemma we referred to in \S\ref{sss.pmp}
and \S\ref{sss.mpm}.
\begin{lemma}
\label{lem.1}
Let $ M $ and $ D $ be the function and domain defined as
\[
M(x,y,z) = y^{z - x} z^{x - y} x^{y - z},\quad 
D := \{0 < z \leq y \leq x\}.
\]
Then $ \sup_D (M) = 1 $ and $ M < 1 $ in the interior of
$ D $. Moreover, for every $ (x,y,z) $ in $ D $, $ M(x,y,z) = 1 $ 
if and only if $ x = y $ or $ y = z $.
\end{lemma}
\begin{proof}
We have
\begin{equation}
\begin{split}
M(x,y,z) = \frac{z^{x - y} x^{y - z}}{y^{x - z}} = 
\left(\frac{z}{y}\right)^{x - y} \left(\frac{x}{y}\right)^{y - z}
= \left[\left(\frac{z}{y}\right)^{\frac{x}{y} - 1} 
\left(\frac{x}{y}\right)^{1 - \frac{z}{y}}\right]^y.
\end{split}
\end{equation}
In order to show that $ M < 1 $ it is enough to 
prove that $ M^{1/y} < 1 $. We substitute 
$ \frac{x}{y} $ with $ a $ and $ \frac{z}{x} $ with $ b $.
Then
\begin{equation}
M^{1/y} = a^{a(1 - b)} b^{a - 1}.
\end{equation}
We define 
\begin{equation}
H(a,b) = \ln(M^{1/y}) = a(1 - b)\ln(a)  + (a - 1)\ln(b).
\end{equation}
We fix $ 0 < b < 1 $ and consider the function $ H(a,b) $
on the interval $ b\leq a\leq \frac{1}{b} $. Clearly,
$ H(1,b) = H(\frac{1}{b},b) = 0 $. Moreover,
\[
\frac{\partial H}{\partial a} (a,b) = 
(1 - b)\ln(a) + (1 - b) + \ln(b).
\]
We have $ \frac{\partial H}{\partial a} (1,b) = (1 - b) + \ln(b) $
which is negative on the interval $ (0,1) $ and has 
$ \partial_a H $ is monotonically increasing
on the interval $ [1,\frac{1}{b}] $. Therefore, $ \partial_a H $
has at most one zero on this interval. In conclusion, 
$ H(\cdot,b) < 0 $ on the interval $ (1,\frac{1}{b}) $ which
implies $ M < 1 $ in the interior of $ D $, while the fact that
$ H(\cdot,b) = 0 $ on the boundary gives the second part of
the statement.
\end{proof}
\begin{center}
\begin{tabular}{|*{10}{c|}}
\hline
$ \var_a $ & $ \var_b $ & $ \var_c $ & 
$ 6 - p $ & $ 6 - q $ & $ 6 - r $ & 
$ \#A $ & $ \Omega $ & Assumptions & Section\\
\hline
\hline
\m & 0 & 0 & \p &  & & 0 
& $ (0,+\infty) $ & (G3) $ \wedge $ (G5) & \ref{ss.pp}\\
\hline
\hline
\m & \p & 0 & &  & & 1 & bounded & (G3) $ \wedge $ (G5) 
& \ref{ss.cp}\\
\hline
\p & \m & 0 & \p & \p & & 0 & $ (0,+\infty) $ 
& $ \neg $(G3) $ \wedge $ (G5) & \ref{ss.cp}\\
\hline
\m & \m & 0 & \p & \p & & 0 & $ (0,+\infty) $ 
& (G3) $ \wedge $ (G5) & \ref{ss.cp}\\
\hline
\hline
\p & \p & \m & \p & \p & \p & 0 & $ (0,+\infty) $
& $ \neg $(G3) $ \wedge $ (G5) & \ref{sss.ppm}\\
\hline
\p & \m & \p &  &  & & 0 & $ (0,+\infty) $
& $ \neg $(G3) $ \wedge $ (G5) & \ref{sss.pmp}\\
\hline
\p & \m & \m & \p & \p & \p & 0 & $ (0,+\infty) $
& $ \neg $(G3) $ \wedge $ (G5) & \ref{sss.pmm}\\
\hline
\hline
\m & \p & \p & \p &  & & 1 & bounded
& (G3) $ \wedge $ (G5) & \ref{sss.mpp}\\
\hline
\m & \p & \p & 0 & \m & \m & 1 & bounded
& (G3) $ \wedge $ (G5) & \ref{sss.mpp}\\
\hline
\m & \p & \p & \m & \m & \m & 1 & bounded
& $ \neg $(G3) $ \wedge $ (G5) & \ref{sss.mpp}\\
\hline
\hline
\m & \p & \m & \p & \p & \p & 0 & $ (0,+\infty) $
& (G3) $ \Rightarrow $ (G5) & \ref{sss.mpm}\\
\hline
\hline
\m & \m & \p & \m & \m & \m & 1 & bounded 
& $ \neg $(G3) $ \wedge $ (G5) & \ref{sss.mmp}\\
\hline
\m & \m & \p & 0 & \m & \m & 1 & bounded 
& $ \neg $(G3) $ \wedge $ (G5) & \ref{sss.mmp}\\
\hline
\m & \m & \p & \p &  & & 1 & bounded
& (G3) $ \wedge $ (G5) & \ref{sss.mmp}\\
\hline
\end{tabular}
\end{center}
\begin{theorem}
For every $ \lambda > 0 $ if the set
$ \mathcal{G}_\lambda\cap H^1 _{r,+} $ is non-empty 
then it is a singleton, provided $ G $
\begin{enumerate}[(i)]
\item is a pure-power with $ \var_a < 0 $
\item is a combined pure-power with $ \var_a < 0 $, or 
$ \var_a > 0 $ in sub-critical regime
\item is a three pure-power combination with 
$ (\var_a,\var_b,\var_c) = (-1,1,1) $ and $ p\leq 6 $ or
\item $ (\var_a,\var_b,\var_c) = (-1,-1,1) $ with $ p < 6 $ or
\item $ (\var_a,\var_b,\var_c) = (-1,1,-1) $ provided
$ r < 6 $ and inequality \eqref{eq.1} holds.
\end{enumerate}
\end{theorem}
\begin{proof}
It follows from \cite[Theorem~1.4]{GG17} or the remarks
made at the introduction of \S\ref{s.uniqueness}.
\end{proof}
Since the mentioned non-linearities satisfy \eqref{G1} and 
\eqref{G2b}, the set $ \mathcal{G}_\lambda\cap H^1 _{r,+} $
is non-empty for every $ \lambda\geq\lambda_* $,
from \cite[Theorem~2]{BBGM07} or 
\cite[Theorem~1.1]{GG17}.
\bibliographystyle{amsplain}
\def\cprime{$'$} \def\cprime{$'$} \def\cprime{$'$} \def\cprime{$'$}
  \def\cprime{$'$} \def\cprime{$'$} \def\cprime{$'$} \def\cprime{$'$}
  \def\cprime{$'$} \def\polhk#1{\setbox0=\hbox{#1}{\ooalign{\hidewidth
  \lower1.5ex\hbox{`}\hidewidth\crcr\unhbox0}}} \def\cprime{$'$}
  \def\cprime{$'$} \def\cprime{$'$}
\providecommand{\bysame}{\leavevmode\hbox to3em{\hrulefill}\thinspace}
\providecommand{\MR}{\relax\ifhmode\unskip\space\fi MR }
\providecommand{\MRhref}[2]{%
  \href{http://www.ams.org/mathscinet-getitem?mr=#1}{#2}
}
\providecommand{\href}[2]{#2}

\end{document}